\long\def\skipit#1{} 

\newcommand{\lra}{\longrightarrow}

\newcommand{\mdef}[1]{\textit{\textbf{#1}}}
\newcommand{\noi}{\noindent}

\newcommand{\ov}{\overline}

\newcommand{\uth}{^{\rm th}}

\newcommand{\vsb}{\vskip-6pt}

\documentclass{amsart}[12pt]

\usepackage{amsfonts}
\usepackage{amsmath}
\usepackage{amssymb}
\usepackage{enumerate}
\usepackage{ifthen,calc}
\usepackage{color}
\usepackage{epsfig}
\usepackage[mathscr]{eucal}
\usepackage{graphicx}
\usepackage{latexsym}
\usepackage{multicol}
\usepackage[format=hang, margin=10pt]{caption}

\graphicspath{{VA-figs}}

\newcounter{hours}
\newcounter{minutes}
\newcommand{\printtime}{
	\setcounter{hours}{\time/60}%
	\setcounter{minutes}{\time-\value{hours}*60}
	\ifthenelse{\value{hours}<10}{0}{}\thehours:%
	\ifthenelse{\value{minutes}<10}{0}{}\theminutes}

\numberwithin{equation}{section}
\numberwithin{figure}{section}
\numberwithin{table}{section}

\newtheorem{thm}{Theorem}[section]

\newtheorem{J-com}{JG-comment}[section]

\theoremstyle{definition}
\newtheorem{example}{Example}[section]

\usepackage{xcolor}

\definecolor{darkgreen}{rgb}{0.2,0.6,0.2}

\begin{document}

\title[Log-Concavity of Genus Polynomials of Ringel Ladders]{{Log-Concavity of the Genus Polynomials of Ringel Ladders}}

\author[J.L. Gross, T. Mansour, T.W. Tucker, and D.G.L. Wang]{Jonathan L. Gross
}
\address{
Department of Computer Science  \\
Columbia University, New York, NY 10027, USA; \newline
email: gross@cs.columbia.edu
}
\author[]{Toufik Mansour
}
\address{
Department of Mathematics  \\
University of Haifa, 31905 Haifa, Israel;  \newline
email: tmansour@univ.haifa.ac.il}
\author[]{Thomas W. Tucker
}
\address{
Department of Mathematics  \\
Colgate University, Hamilton, NY 13346, USA; \newline
email: ttucker@colgate.edu
}
\author[]{David G.L. Wang
}
\address{
School of Mathematics and Statistics  \\
Beijing Institute of Technology, 102488 Beijing, P. R. China;  \newline
email: glw@bit.edu.cn}

\date{}

\begin{abstract}
A \textit{Ringel ladder} can be formed by a \textit{self-bar-amalgamation} operation on a \textit{symmetric ladder}, that is, by joining the root vertices on its end-rungs.  The present authors have previously derived criteria under which \textit{linear chains} of copies of one or more graphs have log-concave genus polynomials.   Herein we establish \textit{Ringel ladders} as the first significant \textit{non-linear} infinite family of graphs known to have log-concave genus polynomials.  We construct an algebraic representation of self-bar-amalgamation as a matrix operation, to be applied to a vector representation of the \textit{partitioned genus distribution} of a symmetric ladder.   Analysis of the resulting genus polynomial involves the use of Chebyshev polynomials. This paper continues our quest to affirm the  quarter-century-old conjecture that all graphs have log-concave genus polynomials.
\end{abstract}


\maketitle                   

\section{\large{Genus Polynomials}\label{sec:intro}}  
\enlargethispage{-12pt}

Our graphs are implicitly taken to be connected, and our \mdef{graph embeddings} are cellular and orientable.  For general background in topological graph theory, see \cite{GrTu87, BW09B}.  Prior acquaintance with the concepts of \textit{partitioned genus distribution} (abbreviated here as \mdef{pgd}) and \textit{production} (e.g., \cite{GKP10, PKG10}) are necessary preparation for reading this paper.  The exposition here is otherwise intended to be accessible both to graph theorists and to combinatorialists.

The number of combinatorially distinct embeddings of a graph~$G$ in the orientable surface of genus~$i$ is denoted by $g_i(G)$.  The sequence $g_0(G)$, $g_1(G)$, $g_2(G)$, $\ldots$, is called the  \mdef{genus distribution} of  $G$.
A genus distribution contains only finitely many positive numbers, and  there are no zeros between the first and last positive numbers.  The \mdef{genus polynomial} is the polynomial
$$\Gamma_G(x) \,=\, g_0(G) + g_1(G)x +g_2(G)x^2 +\ldots\, .$$

\begin{center}
\medskip
\textsc{Log-concave sequences}  
\end{center}

A sequence $A=(a_k)_{k=0}^n$ is said to be \mdef{nonnegative}, if $a_k\ge0$ for all~$k$.
An element~$a_k$ is said to be an \mdef{internal zero} of~$A$ if  $a_k=0$ and if there exist indices $i$ and $j$ with  $i<k<j$, such that $a_ia_j\ne 0$. If $a_{k-1}a_{k+1} \le a_k^2$ for all $k$, then $A$ is said to be \mdef{log-concave}.  If there exists an index~$h$ with $0\le h\le n$ such that
\[
a_0\,\le\, a_1\,\le\, \cdots\,\le\, a_{h-1}\,\le\, a_h\,\ge\, a_{h+1}\,\ge\,\cdots\,\ge\, a_n,
\]
then $A$ is said to be \mdef{unimodal}.  It is well-known that any nonnegative log-concave sequence without internal zeros is unimodal, and that any nonnegative unimodal sequence has no internal zeros. A prior paper \cite{GMTW13} by the present authors provides additional contextual information regarding log-concavity and genus distributions.

For convenience, we sometimes abbreviate the phrase ``log-concave genus distribution'' as \mdef{LCGD}.  Proofs that closed-end ladders and doubled paths have LCGDs \cite{FGS89} were based on explicit formulas for their genus distributions.  Proof that bouquets have LCGDs \cite{GrRoTu89} was based on a recursion.  A conjecture that all graphs have LCGDs was published by \cite{GrRoTu89}.

Stahl's method \cite{Stah91,Stah97} of representing what we have elsewhere formulated as simultaneous recurrences \cite{FGS89} or as a transposition of a \textit{production system} for a surgical operation on graph embeddings as a matrix of polynomials can simplify a proof that a family of graphs has log-concave genus distributions, without having to derive the genus distribution itself.

Newton's theorem that real-rooted polynomials with non-negative coefficients are log-concave is one way of getting log-concavity.  Stahl \cite{Stah97} made the general conjecture (Conjecture 6.4) that all genus polynomials are real-rooted, and he gave a collection of specific test families.  Shortly thereafter, Wagner \cite{Wag97} proved that the genus distributions for the related closed-end ladders and various other test families suggested by \cite{Stah97} are real-rooted. However, Liu and Wang \cite{LW07} answered Stahl's general conjecture in the negative, by exhibiting a chain of copies of the wheel graph $W_4$, one of Stahl's test families, that is not real-rooted.  Our previous paper \cite{GMTW13} proves, nonetheless, that the genus distribution of every graph in the $W_4$-linear sequence is log-concave.  Thus, even though Stahl's proposed approach to log-concavity via roots of genus polynomials is sometimes infeasible, results in \cite{GMTW13} do support Stahl's expectation that chains of copies of a graph are a relatively accessible aspect of the general LCGD problem.  The genus distributions for the family of Ringel ladders, whose log-concavity is proved in this paper, are not real-rooted either. 

Log-concavity of genus distributions for directed graph embeddings has been studied by \cite{BCMM02} and \cite{CGH14}.  Another related area is the continuing study of maximum genus of graphs, of which \cite{KoSk12} is an example.

\enlargethispage{24pt}

\begin{center}
\medskip
\textsc{Linear, ringlike, and tree-like families}  
\end{center}

Stahl used the term``$H$-linear'' to describe chains of graphs that are constructed by amalgamating copies of a fixed graph $H$.  Such amalgamations are typically on a pair of vertices, one in each of the amalgamands, or on a pair of edges.  It seems reasonable to generalize the usage of \mdef{linear} in several ways, for instance, by allowing graphs in the chain to be selected from a finite set.

We use the term \mdef{ring-like} to describe a graph that results from any of the following topological operations on a doubly rooted linear chain with one root in the first graph of the chain and one in the last graph:
\begin{enumerate}
\item a self-amalgamation of two root-vertices;
\item a self-amalgamation of two root-edges;
\item joining one root-vertex to the other root-vertex (which is called a \mdef{self-bar-amalgamation}).
\end{enumerate}

Every graph can be regarded as \mdef{tree-like} in the sense of tree decompositions.  However, we use this term only when a graph is not linear or ring-like.  For any fixed tree-width $w$ and fixed maximum degree $\Delta$, there is a quadratic-time algorithm \cite{Gr12b} to calculate the genus polynomial of graphs of parameters $w$ and $\Delta$.  One plausible approach to the general LCGD conjecture might be to prove it for fixed tree-width and fixed maximum degree.  Recurrences have been given for the the genus distributions of cubic outerplanar graphs \cite{Gr11b}, 4-regular outerplanar graphs \cite{PKG11}, and cubic Halin graphs \cite{Gr13}, all three of which are tree-like.  However, none of these genus distributions have been proved to be log-concave.   Nor have any other tree-like graphs been proved to have LCGDs.

This paper is organized as follows.   Section 2 describes a representation of partitioning of the genus distribution into ten parts as a \textit{pgd-vector}.   Section 3 describes how \textit{productions} are used to describe the effect of a graph operation on the pgd-vector.  Section 4 analyzes how self-bar amalgamation affects the genus distribution.  Section 5 offers a new derivation of the genus distributions of  the Ringel ladders and proof that these genus distributions are log-concave.

\medskip
\section{\large{Partitioned Genus Distributions}\label{sec:PGD}}  

A fundamental strategy in the calculation of genus distributions, from the outset \cite{FGS89}, has been to partition a genus distribution according to the incidence of face-boundary walks on one or more roots.   We abbreviate ``face-boundary walk'' as \mdef{fb-walk}.  For a graph $(G,u,s)$ with two 2-valent root-vertices, we can partition the number $g_i(G)$ into the following four parts:

\begin{description}
\item [$dd_i(G)$] the number of embeddings of $(G,u,v)$ in the surface $S_i$ such that two distinct fb-walks are incident on root $u$ and two on root $v$;
\item [$ds_i(G)$] the number of embeddings in $S_i$ such that two distinct fb-walks are incident on root $u$ and only one on root $v$;
\item [$sd_i(G)$] the number of embeddings in $S_i$ such that one fb-walk is twice incident on root $u$ and two distinct fb-walks are incident on root $v$;
\item [$ss_i(G)$] the number of embeddings in $S_i$ such that one fb-walk is twice incident on root $u$ and one is twice incident on root $v$.
\end{description}

\noi Clearly, $g_i(G) = dd_i(G) + ds_i(G) + sd_i(G) + ss_i(G)$.  Each of the four parts is sub-partitioned:

\begin{description}
\item [$dd^0_i(G)$] the number of type-$dd$ embeddings  of $(G,u,v)$ in $S_i$ such that neither fb-walk  incident at root $u$ is incident at root $v$;
\item [$dd'_i(G)$] the number of type-$dd$ embeddings in $S_i$ such that one fb-walk  incident at root $u$ is incident at root $v$;
\item [$dd''_i(G)$] the number of type-$dd$ embeddings  in $S_i$ such that both fb-walks  incident at root $u$ are incident at root $v$;
\item [$ds^0_i(G)$] the number of type-$ds$ embeddings in $S_i$ such that neither fb-walk incident at root $u$ is incident at root $v$;
\item [$ds'_i(G)$] the number of type-$ds$ embeddings in $S_i$ such that one fb-walk incident at root $u$ is incident at root $v$;
\item [$sd^0_i(G)$] the number of type-$sd$ embeddings in $S_i$ such that the fb-walk incident at root $u$ is not incident on root $v$;
\item [$sd'_i(G)$] the number of type-$sd$ embeddings in $S_i$ such that the fb-walk at root $u$ is also incident at root $v$;
\item [$ss^0_i(G)$] the number of type-$ss$ embeddings in $S_i$ such that the fb-walk incident at root $u$ is not incident on root $v$;
\item [$ss^1_i(G)$] the number of type-$ss$ embeddings in $S_i$ such that the fb-walk incident at root $u$ is incident at root $v$, and the incident pattern is $uuvv$;
\item [$ss^2_i(G)$] the number of type-$ss$ embeddings in $S_i$ such that the fb-walk incident at root $u$ is incident at root $v$, and the incident pattern is $uvuv$.
\end{description}

\noi We define the \mdef{pgd-vector} of the graph$(G,u,v)$ to be the vector
$$\begin{matrix}
\big(dd''(G) & dd'(G) & dd^0(G) & ds^0(G) & ds'(G) & & & \\[4pt]
& & & sd^0(G) & sd'(G) & ss^0(G) & ss^1(G) & ss^2(G)\big)
\end{matrix}
$$
with ten coordinates, each a polynomial in $x$.  For instance,
$$ds'(G) \;=\; d\tilde{s_0}(G) \,+\,  ds'_1(G)x \,+\,  ds'_2(G) x^2 \,+\,  \cdots .$$

\medskip
\section{\large{Symmetric Ladders}\label{sec:sym-lad}}  

We define the \mdef{symmetric ladder} $(\ddot L_n,u,v)$ to be the graph obtained from the cartesian product $P_2 \Box P_{n+2}$ by contracting the respective edges at both ends that join a pair of 2-valent vertices and designating the remaining two 2-valent vertices at the ends of the ladder  as root-vertices.  The symmetric ladders  $(\ddot L_1,u,v)$  and  $(\ddot L_2,u,v)$ are illustrated in Figure \ref{fig:sym-ladders}.  The location of the roots of a symmetric ladder at opposite ends causes it to have a different partitioned genus distribution from other ladders to which it is isomorphic when the roots are disregarded.

\begin{figure} [ht]
\centering
     \includegraphics[width=3.5in]{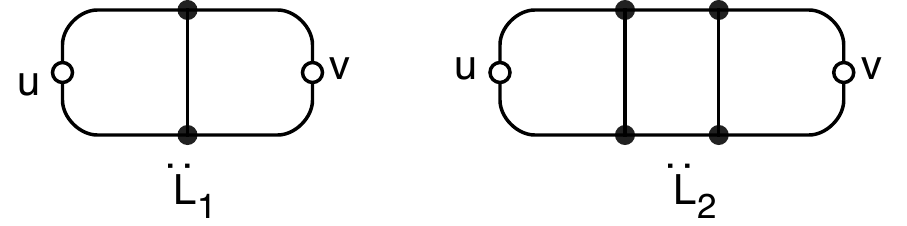} \vsb
\caption{The symmetric ladders $\ddot L_1$ and $\ddot L_2$.}
\label{fig:sym-ladders}
\end{figure}  

\enlargethispage{24pt}
\begin{center}
\medskip
\textsc{Productions}  
\end{center}

A \mdef{production} is an algebraic representation of the set of possible effects of a graph operation on a graph embedding.   For instance, adding a rung to an embedded symmetric ladder  $(\ddot L_n,u,v)$ involves inserting a new vertex on each side of the root-vertex $v$ and then joining the two new vertices.  Since both the resulting new vertices are trivalent, the number of embeddings of  $(\ddot L_{n+1},u,v)$ that can result is 4.  Thus, the sum of the coefficients in the consequent of the production (the right side) is 4.  Figures \ref{fig:ddL-prods1} and~\ref{fig:ddL-prods2} are topological derivations of the following ten productions used to derive the partitioned genus distribution of  $(\ddot L_{n+1},u,v)$ from the partitioned genus distribution of  $(\ddot L_{n},u,v)$.

{\allowdisplaybreaks
\begin{eqnarray*}
\noalign{\vskip-12pt}
dd^0_i &\lra& 2dd^0_i + 2sd^0_{i+1} \\
dd'_i &\lra& dd^0_i + dd'_i + 2sd'_{i+1} \\
dd''_i &\lra& 2dd'_i + 2ss^2_{i+1} \\
ds^0_i &\lra& 2ds^0_i + 2ss^0_{i+1} \\
ds'_i &\lra& ds^0_i + ds'_i + 2ss^1_{i+1} \\
sd^0_i &\lra& 4dd^0_i \\
sd'_i &\lra& 4dd'_i \\
ss^0_i &\lra& 4ds^0_i \\
ss^1_i &\lra& 4ds'_i \\
ss^2_i &\lra& 2ds'_i + 2dd''_i
\end{eqnarray*}
}

\begin{figure} [ht]
\centering
     \includegraphics[width=6in]{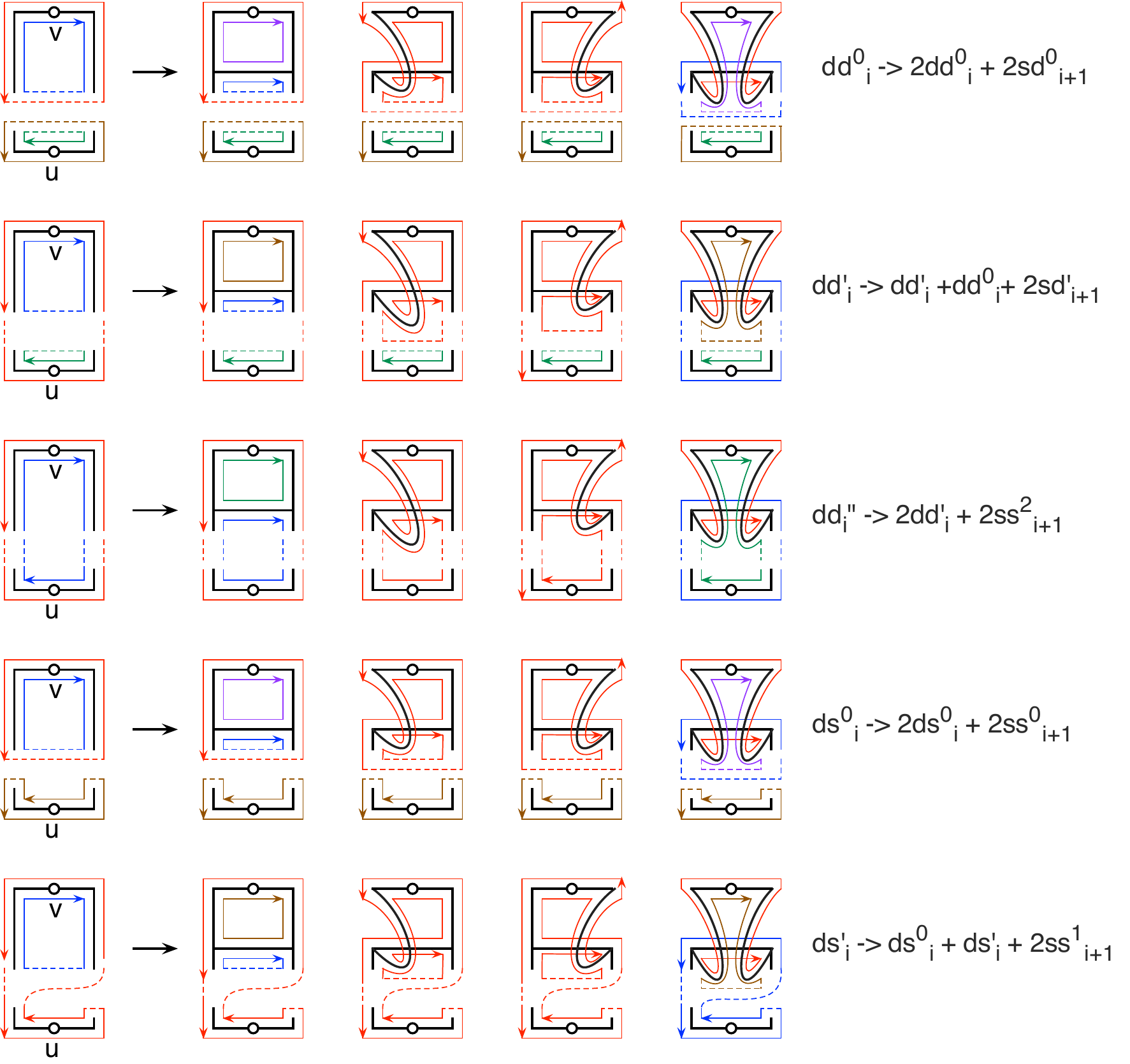}
\caption{\hbox{Five productions for construction of symmetric ladders.}}
\label{fig:ddL-prods1}
\end{figure}  
\clearpage

\begin{figure} [ht]
\centering
     \includegraphics[width=5.7in]{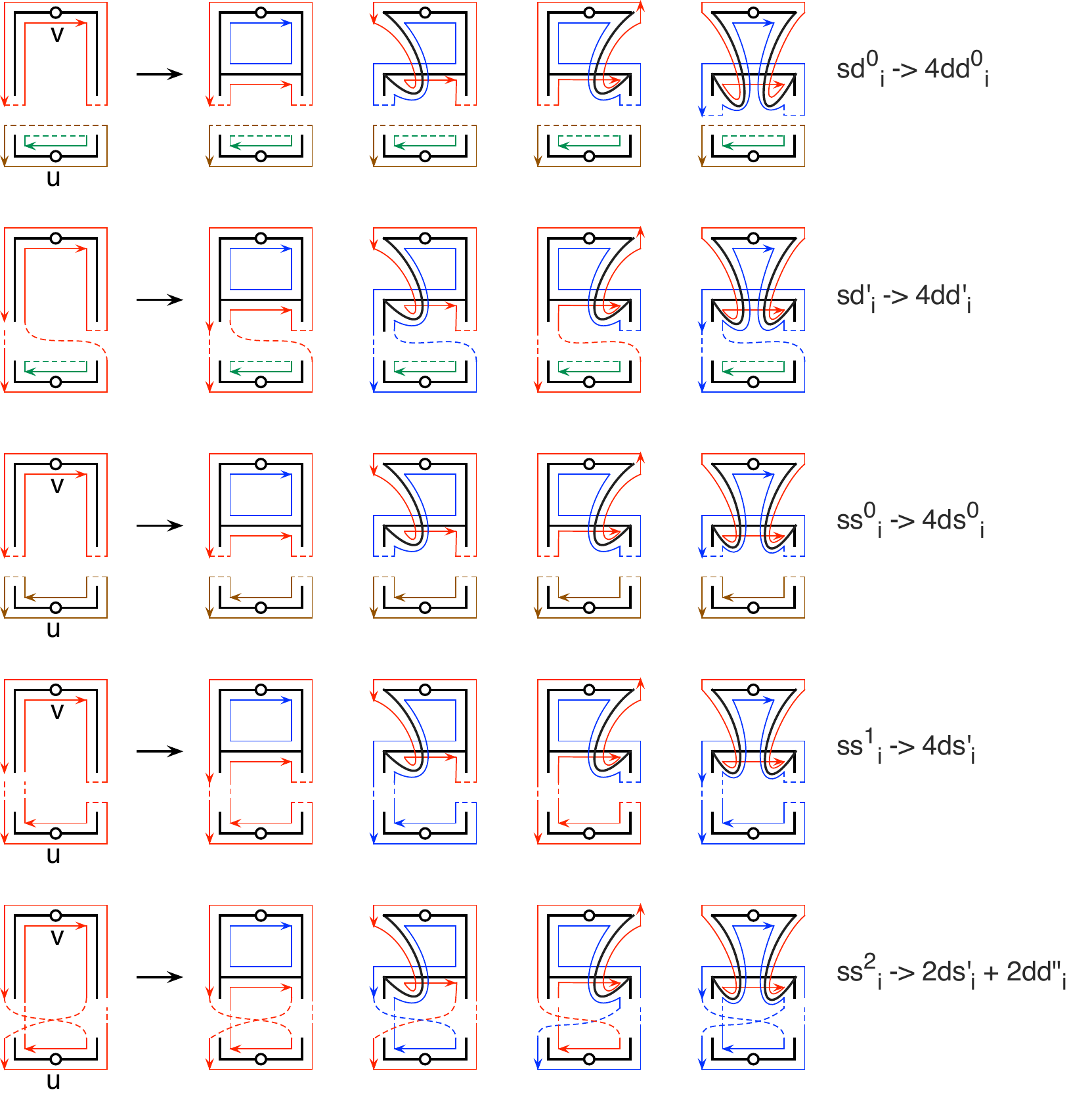}  \vsb
\caption{Five more productions for symmetric ladders.}
\label{fig:ddL-prods2}
\end{figure}  

\newpage

\begin{thm}
The pgd-vector of the symmetric ladder $(\ddot L_0,u,v)$ is
\begin{equation} \label{vec:dd0-unit}
V_{L_0} \,=\, \begin{pmatrix} 0 & 0 & 1 & 0 & 0 & 0 & 0 & 0 & 0 & 0 \end{pmatrix}.
\end{equation}
For $n>0$, the pgd-vector of the symmetric ladder $(\ddot L_n,u,v)$ is the product of the row-vector $V_{L_{n-1}}$ with the $10\times10$ production matrix
\begin{equation} \label{mat:sym}
{\bf M} \;=\; \begin{pmatrix}
2 & 0 & 0 & 0 & 0 & 2x & 0 & 0 & 0 & 0 \\
1 & 1 & 0 & 0 & 0 & 0 & 2x & 0 & 0 & 0 \\
0 & 2 & 0 & 0 & 0 & 0 & 0 & 0 & 0 & 2x \\
0 & 0 & 0 & 2 & 0 & 0 & 0 & 2x & 0 & 0 \\
0 & 0 & 0 & 1 & 1 & 0 & 0 & 0 & 2x & 0 \\
4& 0 & 0 & 0 & 0 & 0 & 0 & 0 & 0 & 0 \\
0 & 4 & 0 & 0 & 0 & 0 & 0 & 0 & 0 & 0 \\
0 & 0 & 0 & 4 & 0 & 0 & 0 & 0 & 0 & 0 \\
0 & 0 & 0 & 0 & 4 & 0 & 0 & 0 & 0 & 0 \\
0 & 0 & 2 & 0 & 2 & 0 & 0 & 0 & 0 & 0
\end{pmatrix}
\end{equation}
\end{thm}

\begin{proof}
Each of the ten rows of the matrix $M$ represents one of the ten productions.  For instance, the first two rows represent the productions
\begin{eqnarray*}
dd^0_i &\lra& 2dd^0_i + 2sd^0_{i+1} \\
dd'_i &\lra& dd^0_i + dd'_i + 2sd'_{i+1} \qedhere
\end{eqnarray*}
\end{proof}

\begin{example}  \label{ex:sym-lad}
We iteratively calculate pgd-vectors of the symmetric ladders $L_n$ for $n\le4$
$$\setcounter{MaxMatrixCols}{14}
\begin{matrix}
V_{L_0} &=&( 0 &   0    &   1 &   0 &   0 &   0 &   0 &   0 &   0 &   0 ) \\
V_{L_1} &=&( 0  &   2     &   0 &   0 &   0 &   0 &   0 &   0 &   0 &   2x ) \\
V_{L_2} &=&( 2 &   2     &   4x & 0 &   4x &   0 & 4x &   0 &   0 &   0) \\
V_{L_3} &=&( 6 &2+24x&   0 &   4x &   4x &   4x &   4x &   0 &  8x^2 & 8x^2) \\
V_{L_4} &=&( 14+40x & 2+40x  & 16x^2 & 12x & 4x+48x^2 & 12x & 4x+48x^2&   8x^2 &  8x^2 &   0 )
\end{matrix}$$
\end{example}

\bigskip
\section{\large{Self-Bar-Amalgamations}\label{sec:SBA}}  

We recall from Section \ref{sec:intro} that the self-bar-amalgamation of any doubly vertex-rooted graph $(G,u,v)$, which is denoted $\ov{*}_{uv}(G,u,v)$, is formed by joining the roots $u$ and $v$.  The present case of interest is when the two roots are 2-valent and non-adjacent.  We observe that if $G$ is a cubic 2-connected graph and if each of the two roots is created by placing a new vertex in the interior of an edge of $G$, then the result of the self-bar amalgamation is again a 2-connected cubic graph.
\smallskip

\begin{thm} \label{thm:SBA}
Let $(G,u,v)$ be a graph with two non-adjacent 2-valent vertex roots.  The (non-partitioned) genus distribution of the graph  $\ov{*}_{uv}(G,u,v)$, obtained by self-bar-amalgamation, can be calculated as the dot-product of the pgd-vector $V_G$ with the following row-vector:
\begin{equation} \label{vec:SBA}
B \;=\; 
\begin{pmatrix}
4x & 1+3x & 2+2x & 4x & 2+2x & 4x & 2+2x & 4x & 4 & 4
\end{pmatrix}.
\end{equation}
\end{thm}
\begin{proof} \vsb
Figures \ref{fig:bar-amalg-D} and \ref{fig:bar-amalg-S} derive the ten corresponding productions.

\enlargethispage{36pt}

{\allowdisplaybreaks
\begin{align*}
\noalign{\vskip-12pt}
dd^0_i &\lra 4g_{i+1} \\
dd'_i &\lra g_i +3g_{i+1} \\
dd'' &\lra 2g_i + 2g_{i+1} \\
ds^0_i &\lra 4g_{i+1} \\
ds'_i &\lra 2g_i + 2g_{i+1} \\
sd^0_i &\lra 4g_{i+1} \\
sd'_i &\lra 2g_i + 2g_{i+1} \\
ss^0_i &\lra 4g_{i+1} \\
ss^1_i &\lra 4g_{i} \\
ss^2_i &\lra 4g_{i}  \qedhere
\end{align*}
} 
\end{proof}
\medskip

\begin{figure} [ht]
\centering
     \includegraphics[width=5.5in]{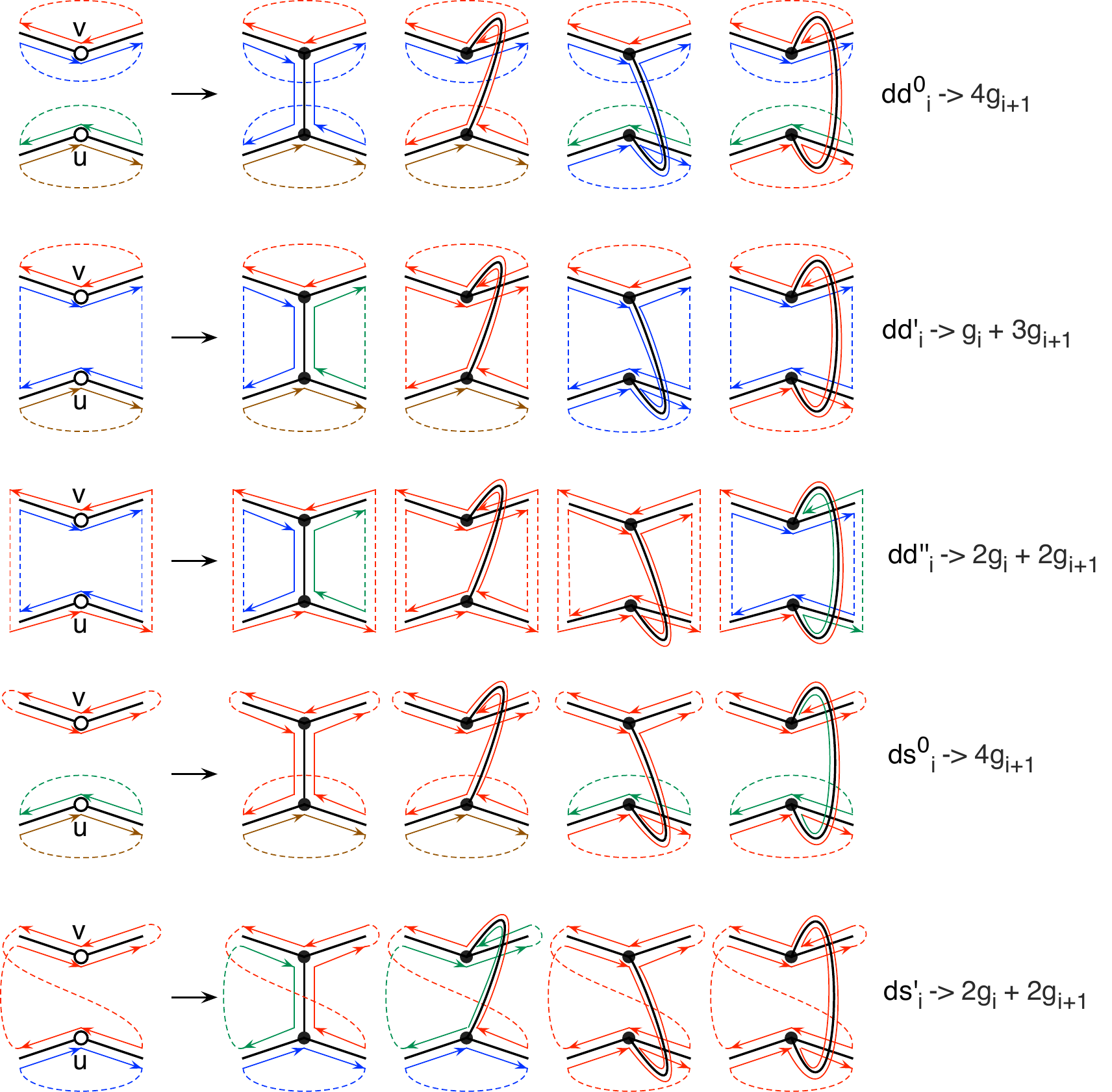}
\caption{Five productions for self-bar-amalgamation.}
\label{fig:bar-amalg-D}
\end{figure}  

\clearpage

\begin{figure} [ht]
\centering
     \includegraphics[width=5.5in]{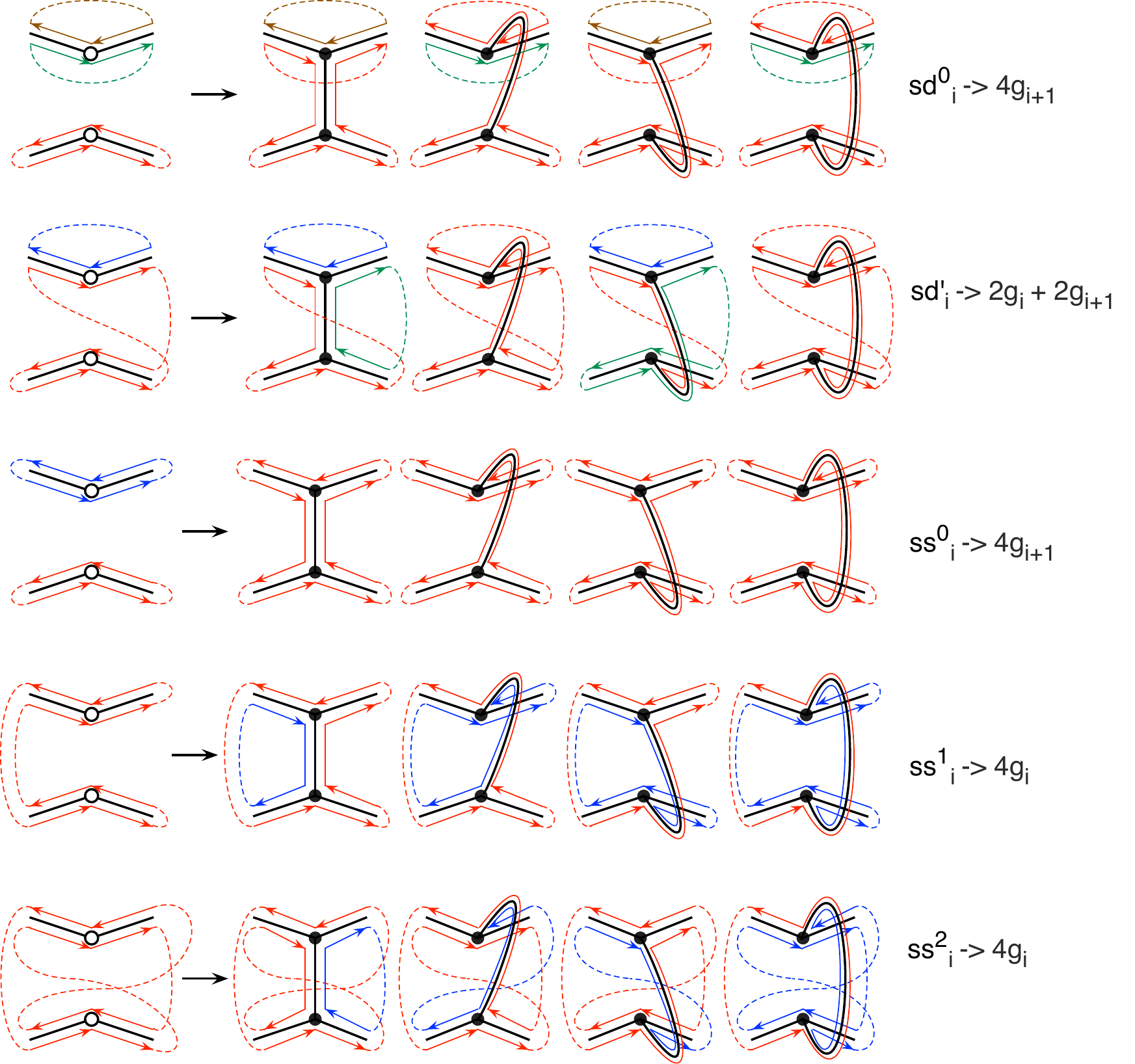}
\caption{Five more productions for self-bar-amalgamation.}
\label{fig:bar-amalg-S}
\end{figure}  


\enlargethispage{36pt}

\smallskip
\section{\large{Ringel Ladders}\label{sec:Ringel}}  

We define a \mdef{Ringel ladder} $RL_{n}$ to be the result of a self-bar-amalgamation on the symmetric ladder $(\ddot L_n, u,v)$.  Such ladders were introduced by Gustin \cite{Gu63} and used extensively by Ringel \cite{Ri74} in his solution with Youngs \cite{RiYo68} of the Heawood map-coloring problem. The Ringel ladder $RL_4$ is illustrated in Figure \ref{fig:Ringel}.  The genus distributions of Ringel ladders were first calculated by Tesar \cite{Te00}.  Our rederivation here is to facilitate our proof of their
log-concavity.

\begin{figure} [ht]
\centering
     \includegraphics[width=2.5in]{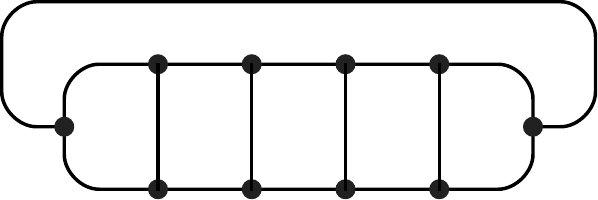}
\caption{The Ringel ladder $RL_4$.}
\label{fig:Ringel}
\end{figure}  

\clearpage

\begin{example}
\noi We take dot products of the pgd-vectors calculated in Example \ref{ex:sym-lad}
$$\setcounter{MaxMatrixCols}{14}
\begin{matrix}
V_{L_0} &=&( 0 &   0    &   1 &   0 &   0 &   0 &   0 &   0 &   0 &   0 ) \\
V_{L_1} &=&( 0  &   2     &   0 &   0 &   0 &   0 &   0 &   0 &   0 &   2x ) \\
V_{L_2} &=&( 2 &   2     &   4x & 0 &   4x &   0 & 4x &   0 &   0 &   0) \\
V_{L_3} &=&( 6 &2+24x&   0 &   4x &   4x &   4x &   4x &   0 &  8x^2 & 8x^2) \\
V_{L_4} &=&( 14+40x & 2+40x  & 16x^2 & 12x & 4x+48x^2 & 12x & 4x+48x^2&   8x^2 &  8x^2 &   0 )
\end{matrix}$$
with the vector \eqref{vec:SBA}
\begin{equation*}
B \;=\;
\begin{pmatrix}
4x & 1+3x & 2+2x & 4x & 2+2x & 4x & 2+2x & 4x & 4 & 4
\end{pmatrix}
\end{equation*}
to obtain the genus polynomials of the corresponding Ringel ladders.
\begin{align*}
\Gamma_{RL_0}(x) &\;=\, 2+2x \\
\Gamma_{RL_1}(x) &\;=\, 2+14x \\
\Gamma_{RL_2}(x) &\;=\, 2+38x+24x^2 \\
\Gamma_{RL_3}(x) &\;=\, 2+70x+184x^2 \\
\Gamma_{RL_4}(x) &\;=\, 2+118x+648x^2 + 256x^3
\end{align*}
\end{example}
\smallskip

\begin{thm}\label{thGRL}
The genus distribution of the Ringel ladder $RL_n$  is given by taking the dot product of the vector $B$ with the product of the vector $V_{L_0}$  and the matrix $M^n$, where $B$ is given by~(\ref{vec:SBA}), and $M$ is given by \eqref{mat:sym}.
\end{thm}
\begin{proof}\vsb
This follows immediately from Theorem \ref{thm:SBA}.
\end{proof}

To deduce an explicit expression of $\Gamma_{RL_n}(x)$, we shall use Chebyshev polynomials.
\mdef{Chebyshev polynomials of the second kind} are defined by
the recurrence relation
$$U_p(x) \;=\; 2xU_{p-1}(x)-U_{p-2}(x),$$
with $U_0(x)=1$ and $U_1(x)=2x$.
It can be equivalently defined by the generating function
\begin{equation}\label{def:U}
\sum_{p\ge0}U_p(x){t^p}={1\over 1-2xt+t^2}.
\end{equation}•
The $p\uth$ Chebyshev polynomial $U_p(x)$ can be expressed  by $$U_p(x)=\sum_{j\geq0}(-1)^j\binom{p-j}{j}(2x)^{p-2j}.$$

\begin{thm}\label{thRL1}
The genus distribution of the Ringel ladder $RL_n$ is given by
\begin{align*}
\Gamma_{RL_n}(x)&\;=\; (1-x)\sum_{j\geq0}\left(\binom{n-j}{j}+\binom{n-j+1}{j}\right)(8x)^j\\
&\qquad+x2^{n+1}\sum_{j\geq0}\left(\binom{n-j}{j}+\binom{n-j+1}{j}\right)(2x)^j.
\end{align*}
\end{thm}
\begin{proof}
Using Theorem \ref{thGRL} and mathematical software such as Maple, we calculate the generating function
\begin{align*}
\sum_{n\geq0}V_{L_0}M^nt^n=(a,b,c,2xta,2xtb,2xta,2xtb,4x^2t^2a,4x^2t^2b,2xtc),
\end{align*}
where 
\begin{align*}
a&=\frac{2t^2}{(1-2t-8xt^2)(1-t-8xt^2)(1-4xt^2)},\\ 
b&=\frac{2t}{(1-t-8xt^2)(1-4xt^2)},\\
c&=\frac{1}{1-4xt^2}.
\end{align*}
This implies 
\begin{align*}
\sum_{n\geq0}\Gamma_{RL_n}(x)t^n&=\sum_{n\geq0}V_{L_0}M^nB^Tt^n\\
&=V_{L_0}(1-tM)^{-1}B^T\\[4pt]
&=\frac{2(1-x)(1+4xt)}{1-t-8xt^2}+\frac{4x(1+2xt)}{1-2t-8xt^2}.
\end{align*}
From Definition~(\ref{def:U}), we can denote the coefficient $\Gamma_{RL_n}(x)$ of $t^n$
in the above generating function in the following form.
\begin{align*}
\Gamma_{RL_n}(x)&=(1-x)\sqrt{-8x}^{n+1}\biggl(\frac{2}{\sqrt{-8x}}U_n\left(\frac1{2\sqrt{-8x}}\right)-U_{n-1}\left(\frac1{2\sqrt{-8x}}\right)\biggr)\\
&\qquad+x\sqrt{-8x}^{n+1}\biggl(\frac{4}{\sqrt{-8x}}U_n\left(\frac1{\sqrt{-8x}}\right)-U_{n-1}\left(\frac1{\sqrt{-8x}}\right)\biggr)\\
&=(1-x)\sum_{j\geq0}\left(2\binom{n-j}{j}+\binom{n-j}{j-1}\right)(8x)^j\\
&\qquad+x\sum_{j\geq0}\left(2\binom{n-j}{j}+\binom{n-j}{j-1}\right)2^{n+1+j}x^j.
\end{align*}
Using the Pascal recursion ${n\choose k}+{n\choose k-1}={n+1\choose k}$,
we get the desired expression.
\end{proof}
\medskip

\begin{thm} \label{thm:RL-LC}
The Ringel ladders $RL_n$ have log-concave genus distributions.
\end{thm}

\begin{proof} \vsb
Let $a_{n,j}$ be the coefficient of $x^j$ of $\Gamma_{RL_n}(x/2)$. By Theorem~\ref{thRL1}, we have
\begin{align*}
a_{n,j} &= \left[\binom{n-j}{j}+\binom{n-j+1}{j} - {1\over 8}\binom{n-j+1}{j-1} 
	- {1\over 8}\binom{n-j+2}{j-1}\right]4^j\\
&\qquad + 2^n\left[\binom{n-j+1}{j-1} + \binom{n-j+2}{j-1}\right].
\end{align*}
Note that $a_{n,j}=0$ for $j\ge\lfloor n/2\rfloor+2$.  We define $f_n(j)=a_{n,j}^2-a_{n,j-1}a_{n,j+1}$.  When $j=\lfloor n/2\rfloor+1$, we have $a_{n,j+1}=0$ and thus $f_n(j)=a_{n,j}^2\ge0$.
So it suffices to show that
\begin{equation} \label{goal}
f_n(j)\geq0\qquad\text{for all $n\geq2$ and $1\le j\le n/2$.}
\end{equation}
Using Maple, it is routine to verify that Inequality~(\ref{goal}) holds true for $n<100$.
We now suppose that $n\geq100$, and we define 
\begin{equation}\label{def:g}
g_n(j) \;=\; f_n(j)\cdotp\frac{64\,j!\,(j+1)!\,(n-2j+5)!\,(n-2j+3)!}{(n-j)!\,(n-j-1)!}.
\end{equation}
We employ the expression \eqref{def:g} because it can be written, if one replaces $j$ by $x$, in the form
\begin{equation}\label{g}
g_n(x) \;=\; 16^xs_2+2^{n+2x+1}x(n-x+1) \bigl(s_1+2^{n-2x}s_0\bigr),
\end{equation}
where $s_2$, $s_1$ and $s_0$ are polynomials in $n$ and $x$ as follows:
{\allowdisplaybreaks
\begin{multline*}
s_2 \;=\; 256n(n+5)(n+4)(n+3)^2(n+2)^2(n+1)^2  \\
-4(n+3)(n+2)(n+1)\cdot\\
(848n^5+10503n^4+46749n^3+88974n^2+64168n+7680)x \\
+(19140n^7+303416n^6+1959723n^5+6630515n^4+12527817n^3\\
+12930761n^2+6465660n+1080000)x^2 \\
+(59628n^6+799668n^5+4257252n^4+11406255n^3\\
+15964242n^2+10757127n+2565612
)x^3\\
+(110781n^5 +1228365n^4+5215302n^3\\+10470267n^2+9734049n+3223854
)x^4\\
-(122760n^4+1099197n^3+3570660n^2 +4898043n+2323908)x^5\\
+(75141n^3+542916n^2+1286307n+964224)x^6\\
-(19602n^2+137214n+213840)x^7 \;+\; 19602x^8,
\end{multline*}
} 
\begin{multline*}
s_1  \;=\;   4n(n+4)(n+3)(n+2)(n+1)(184n^2+595n+538)\\
-(288n^7+10832n^6+ 97908n^5+ 388214n^4+ 782118n^3\\
  + 803168n^2+363528n
+39360)x\\
+(3492n^6+66912n^5+417975n^4+1177485n^3\\+1603200n^2+969000n+174744)x^2\\
-(17964n^5+225066n^4+972648n^3+1831368n^2+1476624n+375000)x^3\\
+(50805n^4+445635n^3+1302147n^2+1485999n+534402)x^4\\
-(85266n^3+519912n^2+949644n+510462)x^5\\
+(84861n^2+331209n+293922)x^6-(46332n+88938)x^7+10692x^8,
\end{multline*}
and
\begin{multline*}
{s_0\over 32(x+1)(n-x)} \;=\;  4(n+4)(n+3)(n+2)^2(n+1)\\
-(20n^4+185n^3+616n^2+883n+468)x\\
+(33n^3+222n^2+501n+402)x^2-(18n^2+90n+144)x^3+18x^4.
\end{multline*}

In view of~(\ref{goal}), (\ref{def:g}) and (\ref{g}), it suffices to show that both $s_2$ and $s_1+2^{n-2x}s_0$ are nonnegative for $x\le n/2$.

First, we show that $s_2\ge0$.
Toward this objective, we write $x=kn$. Then $0\le k\le 1/2$. Define $\tilde{s}_2=s_2/n$.
Then $\tilde{s}_2$ is a polynomial of degree~$8$ in~$n$.
For $0\le j\le 8$, define
\[
q_j={d^j\over dn^j}\tilde{s}_2.
\]
Then we have
\[
q_8=40320(1-2k)(3k-2)^2(33k^2-33k+8)^2\ge0.
\]
So $q_7$ is increasing in $n$ for any $0\le k\le 1/2$.
We compute
\begin{multline*}
q_7\big|_{n=4}
=98794080k^8-3852969120k^7+14855037120k^6\\
-25338685680k^5+24057719280k^4-13647130560k^3\\
+4616115840k^2-861376320k+68382720.
\end{multline*}
It is elementary to prove that
\[
q_7\big|_{n=4}>0\qquad\text{for all $k\in[0,1/2]$}.
\]
Alternatively, one may find this positivity by drawing its figure in Maple.
It follows that $q_6$ is increasing in the interval $[4,\infty)$ of $n$. Next, we can compute
\begin{multline*}
q_6\big|_{n=4}
=395176320k^8-9243020160k^7+36108808560k^6\\
-64328152320k^5+64252375200k^4-38420136000k^3\\
+13701665520k^2-2694375360k+225239040.
\end{multline*}
Again, it is routine to prove that
\[
q_6\big|_{n=4}>0\qquad\text{for all $k\in[0,1/2]$}.
\]
So $q_5$ is increasing in $n$ on the interval $[4,\infty)$. Continuing in this bootstrapping way,
we can prove that all $q_4$, $q_3$, $q_2$, $q_1$, $q_0$ are increasing for $n\in[4,\infty)$.
Since
\begin{multline*}
q_0\big|_{n=4}
=321159168k^8-4408639488k^7+20075655168k^6\\
-46290382848k^5+62167349376k^4-50943602304k^3\\
+25184659968k^2-6919073280k+812851200
\end{multline*}
is positive for all $k\in[0,1/2]$, we conclude that $q_0>0$ for all $n\ge4$ and all $k\in[0,1/2]$.
That is, $s_2>0$.

On the other hand, we define
\[
p_n=s_1+2^{n-2x}s_0
\]
It remains to show $p_n\ge0$ for all $x\in[0,n/2]$.
We shall do that for the intervals $[0,n/3]$ and $[n/3,n/2]$, respectively.

For the first interval,
we claim that
\begin{equation}\label{s0pos}
s_0\ge0\quad\text{for all $n\ge100$ and all $0\le x\le n/2$}.
\end{equation}•
We will show~(\ref{s0pos}) by using the same derivative method.
In fact, consider
\[
\tilde{s_0}(x)=\frac{s_0}{32(x+1)(n-x)}.
\]
Note that $\tilde{s_0}(x)$ is a polynomial in $x$ of degree $4$.
For $0\le j\le 4$, denote
\[
\frac{d^j}{dx^j}\tilde{s_0}(x)=\tilde s_0^{(j)}(x).
\]
Since $\tilde s_0^{(4)}(x)=432>0$, the 3rd derivative
\[
\tilde{s_0}^{(3)}(x)=432x-108(n^2+5n+8)
\]
is increasing on the interval $[0,n/2]$.
Since
\[
\tilde{s_0}^{(3)}(n/2)=-108(n^2+3n+8)<0,
\]
we infer that $\tilde{s_0}^{(3)}(x)<0$ for all $x\in[0,n/2]$. So the second derivative
$$\tilde{s_0}^{(2)}(x) = 6(11n^3+74n^2+167n+134)-108(n^2+5n+8)x+216x^2$$
is decreasing. Since
\[
\tilde{s_0}^{(2)}(n/2)=6(2n^3+38n^2+95n+134)>0,
\]
we deduce that $\tilde{s_0}^{(2)}(x)>0$ for all $x$. Therefore,
\begin{multline*}
\tilde{s_0}^{(1)}(x)=-(20n^4+185n^3+616n^2+883n+468)\\
+6(11n^3+74n^2+167n+134)x-54(n^2+5n+8)x^2+72x^3
\end{multline*}
is increasing. Since
\[
\tilde{s_0}^{(1)}(n/2)=-2(n^4+43n^3+446n^2+962n+936)<0,
\]
we find $\tilde{s_0}^{(1)}(x)<0$.
It follows that $\tilde{s_0}(x)$ is decreasing. Since
\[
\tilde{s_0}(n/2)=\frac{1}{8}(7n^4+154n^3+1112n^2+2096n+1536)>0,
\]
we infer that $s_0(x)\geq0$ and this completes the proof for Claim~(\ref{s0pos}).

Now, for $x\in[0,n/3]$, it suffices to prove that $p_1(x)=s_1+2^{n/3}s_0\geq0$.  This can be done by considering derivatives of $p_1(x)$, with respect to $x$, along the same way.  So we omit the proof.   

For the other interval $[n/3,n/2]$, we compute
\begin{multline*}
f_n(n/2)=\frac{4^n}{1474560}(397n^6+9528n^5+102100n^4+619680n^3\\
+2315488n^2+5041152n+5898240)>0.
\end{multline*}
So we can suppose $x\in [n/3,n/2-1]$, i.e., $n\in[2x+2,3x]$.
Define
\[
h_j(n)=\frac{d^j}{dn^j}p_n
\]
Expanding in $n-2-2x$, the function $2^{2x-n}h_8(n)$ can be recast as
\[
2^{2x-n}h_8(n)=\sum_{i=0}^6\sum_{j=0}^{7-i}a_{ij}x^j(n-2-2x)^i,
\]
where $a_{ij}\geq0$.  So $h_8(n)\ge0$ for all $n\in[2x+2,3x]$. 
It is elementary to prove that the univariate function $h_7(2x+2)$ is non-negative.  
Again, it is routine to see this by drawing a graph of the function $h_7$ 
with the aid of Maple.  It follows that $h_7(n)\ge0$ for all $n\in[2x+2,3x]$. 
Then, we check 
with Maple
that $h_6(2x+2)\ge0$, from which it follows that $h_6(n)\ge0$ for all $n\in[2x+2,3x]$.   
Continuing in this way, we can show that, for all $n\in[2x+2,3x]$, we have  
$$\begin{matrix} h_5(n)\ge0, & h_4(n)\ge0, & \cdots, & h_0(n)\ge0\end{matrix}$$
In particular, we have $p_n = h_0(n)\ge0$.  
\end{proof}

\bigskip

\end{document}